\documentclass[10pt]{amsart}

\usepackage{amssymb}
\usepackage{color}
\usepackage{mathtools}
\usepackage[hidelinks]{hyperref} 
\usepackage{enumerate}

\newcommand{\R}{\mathbb{R}}

\newcommand{\N}{\mathbb{N}}
\newcommand{\hd}{\dim_{\textup{H}}}
\newcommand{\bd}{\dim_{\textup{B}}}
\newcommand{\ubd}{\overline{\dim}_{\textup{B}}}
\newcommand{\lbd}{\underline{\dim}_{\textup{B}}}
\newcommand{\uid}{\overline{\dim}_{\,\theta}}
\newcommand{\lid}{\underline{\dim}_{\,\theta}}

\newcommand{\lidone}{\underline{\dim}_{\,1}}

\newcommand{\be}{\begin{equation}}
\newcommand{\ee}{\end{equation}}

\DeclareMathOperator*\lowlim{\liminf}
\DeclareMathOperator*\uplim{\limsup}
\DeclarePairedDelimiter{\ceil}{\lceil}{\rceil}
\newtheorem{theorem}{Theorem}[section]
\newtheorem{lemma}[theorem]{Lemma}
\newtheorem{cor}[theorem]{Corollary}

\newtheorem{question}[theorem]{Question}
\theoremstyle{definition}

\theoremstyle{remark}

\allowdisplaybreaks

\numberwithin{equation}{section}
\begin{document}

\title{Dimensions of fractional Brownian Images}
%    author one information
\author{Stuart A. Burrell}
\address{Stuart A. Burrell, School of Mathematics and Statistics, University of St Andrews, St Andrews, KY16 9SS, United Kingdom.}
\email{sb235@st-andrews.ac.uk}
%\curraddr{}
%\thanks{}

%    \subjclass is required.
\date{\today}
\dedicatory{}

\begin{abstract}
This paper concerns the intermediate dimensions, a spectrum of dimensions that interpolate between the Hausdorff and box dimensions. Potential theoretic methods are used to produce dimension bounds for images of sets under H\"older maps and certain stochastic processes. We apply this to compute the almost-sure value of the dimension of Borel sets under index-$\alpha$ fractional Brownian motion in terms of dimension profiles defined using capacities. As a corollary, this establishes continuity of the profiles for Borel sets and allows us to obtain an explicit condition showing how the Hausdorff dimension of a set may influence the typical box dimension of H\"older images such as projections. The methods used propose a general strategy for related problems; dimensional information about a set may be learned from analysing particular fractional Brownian images of that set. To conclude, we obtain bounds on the Hausdorff dimension of exceptional sets, with respect to intermediate dimensions, in the setting of projections. \\
\end{abstract}

\maketitle

\section{Introduction}

The growing literature on dimension spectra is beginning to provide a unifying framework for the many notions of dimension that arise throughout the field of fractal geometry. Suppose you are given two notions of dimension, $\dim_X$ and $\dim_Y$, with $\dim_X E \leq \dim_Y E$ for all $E \in \R^n$. Dimension spectra aim to provide a continuum of dimensions, perhaps denoted $\dim_\theta$ and parametrised by $\theta \in [0, 1]$, such that $\dim_0 = \dim_X$ and $\dim_1 = \dim_Y$. This is of interest for a number of reasons. For example, $\dim_X$ and $\dim_Y$ may behave differently for certain classes of sets, since each may be sensitive to different geometric properties. Thus, it may be valuable to understand for what $\theta$ this transition in behaviour occurs, potentially deepening our understanding of $\dim_X$, $\dim_Y$, and the family sets in question. Despite their extremely recent introduction, they have already seen surprising applications, for example, \cite[Corollary 6.4]{bufafr:2019} and \cite{spirals:2019}.\\

There are currently two main dimension spectra of interest. For $E \subset \R^n$, recall
$$
\hd E \leq \bd E \leq \dim_A E
$$
where, from left to right, these denote Hausdorff dimension, box dimension and Assouad dimension. Fraser and Yu introduced the Assouad spectrum to form a partial interpolation between the upper box dimension and the Assouad dimension, see \cite{assouad-spec}. The main focus of this paper will be the intermediate dimensions of Fraser, Kempton and Falconer \cite{fafrke:2018} that interpolate between the popular Hausdorff and box dimensions. These will be formally introduced in Section \ref{setting}.\\

In developing this new theory, it is natural to re-examine classical theorems and see how well they adapt to this more general setting. Work along these lines has already begun, with \cite{spirals:2019,frtr:2018,assouad-spec} investigating the Assouad spectrum and \cite{bufafr:2019} establishing a Marstrand-type projection theorem for the intermediate dimensions. This paper generalises \cite{bufafr:2019} beyond projections to general H\"older images and images of sets under stochastic processes, such as index-$\alpha$ fractional Brownian motion. Recall that a map $f: E \rightarrow \R^m$ is $\alpha$-H\"older on $E \subset \R^n$ if there exists $c > 0$ and $0 < \alpha \leq 1$ such that
$$
|f(x) - f(y)| \leq c|x - y|^\alpha
$$
for all $x, y \in E$. This scheme of work continues a tradition of Xiao \cite{shiehxiao,xiao}, who used dimension profiles almost immediately after their introduction in 1997 \cite{faho:1997} to consider the packing dimensions of sets under fractional Brownian motions. Unexpectedly, obtaining bounds on the dimension of fractional Brownian images allowed us to quickly establish continuity of the profiles for arbitrary Borel sets. Moreover, this led to an explicit condition showing how the Hausdorff dimension of a set may influence the typical box dimension of H\"older images such as projections. Both of these applications followed from a method which suggests a more general philosophy; dimensional information in a general setting can be obtained by transporting information back from a well-chosen fractional Brownian image.\\

Finally, we return to the setting of projections where our main results may be applied to bound the Hausdorff dimension of the exceptional sets, see Theorem \ref{excep}. That is, the dimension of the family of sets whose projection has unusually small dimension. There is a long history of interest in this topic, see \cite{fal:1982,mattila1975b,kaufmanproj}. Throughout, we adopt a capacity theoretic approach to intermediate dimension profiles, as in \cite{bufafr:2019}, while adapting this strategy to meld it with ideas from \cite{fal:2018}. 

\section{Setting and Preliminaries}\label{setting}

In this section we will define the necessary tools and concepts used throughout.  This section is intentionally brief, and the interested reader is directed to \cite{bufafr:2019} for a more elaborate discussion of the material and \cite{Falconer} for a gentle introduction to dimension theory. We begin with the precise formulation of the intermediate dimensions. Throughout, all sets are assumed to be non-empty, bounded and Borel.\\

For $E \subset \R^n$ and $0 < \theta \leq 1$, the  {\em lower intermediate dimension} of $E$ may be defined as
\begin{align*}
\lid E =  \inf \big\{& s\geq 0  :  \mbox{ \rm for all $\epsilon >0$ and all $r_0>0$, there exists }  \nonumber\\
&\mbox{ $0<r\leq r_0$ and a cover $ \{U_i\} $ of $E$  such that} \\
 & \mbox{ $r^{1/\theta} \leq  |U_i| \leq r $ and 
 $\sum |U_i|^s \leq \epsilon$}  \big\}\nonumber
\end{align*}
and the corresponding {\em upper intermediate dimension} by
\begin{align*}
\uid E =  \inf \big\{& s\geq 0  :  \mbox{ \rm for all $\epsilon >0$, there exists $r_0>0$ such that} \nonumber\\
& \mbox{for all $0<r\leq r_0$, there is a cover $ \{U_i\} $ of $E$} \\
&\mbox{such that $r^{1/\theta} \leq  |U_i| \leq r$ and 
$\sum |U_i|^s \leq \epsilon$}  \big\},\nonumber
\end{align*}
where $|U|$ denotes the diameter of a set $U \subset \R^n$. If $\theta = 0$, then we recover the Hausdorff dimension in both cases, since the covering sets may have arbitrarily small diameter. Moreover, if $\theta=1$, then we recover the lower and upper box-counting dimensions, respectively, since sets within admissible covers are forced to have equal diameter. While the above makes the interpolation intuitive, for technical reasons it is practical to use an equivalent formulation. First, for bounded and non-empty $E \subset \R^n$, $\theta  \in (0, 1]$ and $s\in [0,n]$, define
\begin{align*}
S_{r, \theta}^s(E) := \inf \Big\{& \sum_i |U_i|^s :\mbox{ \rm $\{U_i\}_i$\textnormal{ is a cover of} $E$ \textnormal{ such that }}\nonumber\\ 
& \mbox{ $r \leq |U_i| \leq r^\theta$\,\,\textnormal{ for all } $i$} \Big\}.
\end{align*}
It is proven in \cite[Section 2]{bufafr:2019} that
$$
\lid E =  \bigg(\textnormal{ the unique } s\in [0,n] \textnormal{ such that  } \liminf\limits_{r \rightarrow 0} \frac{\log S_{r, \theta}^s(E)}{-\log r} =0\bigg)
$$
and 
$$
\uid E =  \bigg(\textnormal{ the unique } s\in [0,n] \textnormal{ such that  } \limsup\limits_{r \rightarrow 0} \frac{\log S_{r, \theta}^s(E)}{-\log r} =0\bigg).
$$

The first step of a capacity theoretic approach is to define an appropriate kernel for the setting. For each collection of parameters $\theta \in (0, 1]$, $t > 0$, $0 \leq s \leq t$ and $0<r <1$, define $\phi_{r, \theta}^{s, t} : \R^n \rightarrow \R$ by
\be\label{ker}
{\phi}_{r, \theta}^{s, t}(x) = \begin{cases} 
      1 & 0\leq |x| < r \\
      \big(\frac{r}{|x|}\big)^s & r\leq |x| < r^\theta   \\
      \frac{r^{\theta(t-s) + s}}{|x|^t}\ & r^\theta \leq |x|
   \end{cases}.
\ee
In addition, for Lemma \ref{hatlem} and Theorem \ref{ctsfam}, we will require a set of modified kernels $\widetilde{\phi}_{r, \theta}^s : \R^m \rightarrow \R$ ($m \in \N$) given by
\be\label{modker}
\widetilde{\phi}_{r, \theta}^s(x)=
\begin{cases}
1 & |x| < r\\ 
\big(\frac{r}{|x|}\big)^s & r \leq |x| \leq r^\theta\\
0 & r^\theta < |x|
\end{cases} ,
\ee
where $0<r<1, \theta\in (0,1]$ and $0<s\leq m$. Using the first of these kernels, we define the \emph{capacity} of a compact set $E \subset \R^n$ to be
\begin{equation*}
{C_{r, \theta}^{s, t}(E)} = \left(\inf\limits_{\mu \in \mathcal{M}(E)} \int\int \phi_{r, \theta}^{s, t}(x - y) \,d\mu(x)d\mu(y)\right)^{-1},
\end{equation*}
where $\mathcal{M}(E)$ denotes the set of probability measures supported on $E$. For a set that may be bounded, but not closed, the capacity is simply defined to be that of its closure. \\

A measure that obtains the infimum in the definition of capacity is known as an \emph{equilibrium measure}. The existence of such measures and the relationship between the minimal energy and the corresponding potentials is standard in classical potential theory. We state this in a convenient form; it is easily proved for continuous kernels, see, for example, \cite[Lemma 2.1]{fal:2019}. 

\begin{lemma}\label{attaincap}
Let $E\subset \R^n$ be compact, $t > 0$, $0 \leq s \leq t$, $\theta \in (0, 1]$  and $0<r<1$. Then there exists an equilibrium measure $\mu \in \mathcal{M}(E)$ such that 
$$\int\int \phi_{r, \theta}^{s, t}(x - y) d\mu(x)d\mu(y) = \frac{1}{C_{r, \theta}^{s, t}(E)} =: \beta.$$
Moreover, 
$$\int \phi_{r, \theta}^{s, t}(x - y) d\mu(y) \geq \beta$$
for all $x \in E$, with equality for $\mu$-almost all $x \in E$.
\end{lemma}

In \cite{bufafr:2019} a close relationship between the capacity of a set $E$ and $S_{r, \theta}^s(E)$ is established, see \cite[Proposition 4.2]{bufafr:2019}. This connection allowed \emph{intermediate dimension profiles} to be introduced, which in turn are central to a Marstrand-type projection theorem \cite[Theorem 5.1]{bufafr:2019}. For $t > 0$, we define the \emph{lower intermediate dimension profile} of $E \subset \R^n$ as
\be\label{lidp}
\lid^t E =  \bigg(\textnormal{ the unique } s\in [0,t] \textnormal{ such that  } \lowlim\limits_{r \rightarrow 0}\frac{\log C_{r, \theta}^{s, t}(E)}{-\log r} = s\bigg)
\ee
and the \emph{upper intermediate dimension profile} as
\be\label{ludp}
\uid^t E =  \bigg(\textnormal{ the unique } s\in [0,t] \textnormal{ such that  } \uplim\limits_{r \rightarrow 0}\frac{\log C_{r, \theta}^{s, t}(E)}{-\log r} = s\bigg).
\ee
In \cite{bufafr:2019}, only integer $t \leq n$ was required, as this corresponded to the topological dimension of the subspace being projected onto. However, as we shall see, it is necessary and possible to consider dimension profiles for non-integer and arbitrarily large $t$ in the more general setting of Theorems \ref{holderthm}, \ref{ctsfam} and \ref{frcbrown}. In fact, to ensure that the above profiles exist, we require the following lemma, which allows \cite[Lemma 3.2]{bufafr:2019} to be easily extended for this greater range of $t$.

\begin{lemma}\label{existlem}
For bounded $E \subset \R^n$ and all $t > 0$,
$$
\lowlim\limits_{r \rightarrow 0}\frac{\log C_{r, \theta}^{t, t}(E)}{-\log r} - t \leq \uplim\limits_{r \rightarrow 0}\frac{\log C_{r, \theta}^{t, t}(E)}{-\log r} - t \leq 0.
$$
In particular, there exists a unique $s \in [0, t]$ such that
$$
\lowlim\limits_{r \rightarrow 0}\frac{\log C_{r, \theta}^{s, t}(E)}{-\log r} = s
$$
and unique $s' \in [0, t]$ such that
$$
\uplim\limits_{r \rightarrow 0}\frac{\log C_{r, \theta}^{s', t}(E)}{-\log r} = s'.
$$
\end{lemma}
\begin{proof}
It suffices to show that
\begin{equation}\label{criterion}
C_{r, \theta}^{t, t}(E) \leq c r^{-t}
\end{equation}
for some fixed $c > 0$ depending only on $E$ and $t$. For $0 < r < 1$, let $\mu$ be the equilibrium measure associated with $\phi_{r, \theta}^{t, t}$. Since $E$ is bounded, there exists a constant $B > 1$ such that
$$
|x - y | \leq B
$$
for all $x, y \in E$. 
Directly from the definition,
\begin{align*}
\phi_{r, \theta}^{t, t}(x - y) &=
\begin{cases} 
      1 & 0\leq |x - y| < r \\
      \big(\frac{r}{|x - y|}\big)^t & r\leq |x - y|
   \end{cases} \\
&\geq B^{-t} r^t.
\end{align*}
for all $x, y \in E$. Hence, 
\begin{align*}
\int \int \phi_{r, \theta}^{t, t}(x- y) \,d\mu(x)d\mu(y) \geq B^{-t}r^t,
\end{align*}
from which (\ref{criterion}) follows. The final part of the lemma may then be deduced since
$$
\lowlim\limits_{r \rightarrow 0}\frac{\log C_{r, \theta}^{0, t}(E)}{-\log r} - 0 \geq 0,
$$
and
$
\lowlim\limits_{r \rightarrow 0}\frac{\log C_{r, \theta}^{s, t}(E)}{-\log r} - s
$
is continuous and strictly monotonically decreasing in $s$ by a trivial extension to \cite[Lemma 3.2]{bufafr:2019}. We may similarly argue for the upper limits.
\end{proof}

To conclude this section, we briefly recall the definition of index-$\alpha$ fractional Brownian motion ($0 < \alpha < 1$), which we denote $B_\alpha : \R^n \rightarrow \R^m$ for $m, n \in \N$. In particular, $B_{\alpha} = (B_{\alpha, 1}, \dots, B_{\alpha, m})$, where for each $B_{\alpha, i} : \R^n \rightarrow \R$:
\begin{enumerate}[i)]
\item $B_{\alpha, i}(0) = 0$;
\item $B_{\alpha, i}$ is continuous with probability $1$;
\item the increments $B_{\alpha, i}(x) - B_{\alpha, i}(y)$ are normally distributed with mean $0$ and variance $|x-y|^{2\alpha}$ for all $x, y \in \R^n$.
\end{enumerate}
Moreover, $B_{\alpha, i}$ and $B_{\alpha, j}$ are independent for all $i, j \in \{1,\dots, m\}$. It immediately follows that for Borel $A\subset \R$,
$$
\mathbb{P}(B_{\alpha, i}(x) - B_{\alpha, i}(y) \in A) = \frac{1}{\sqrt{2\pi}}\frac{1}{|x-y|^\alpha}\int\limits_{t \in A} \exp\left(\frac{-t^2}{2|x-y|^{2\alpha}}\right)\,dt.
$$
The reader may enjoy the classical text of Kahane \cite{kahane:book} for a more detailed account of index-$\alpha$ fractional Brownian motion and related stochastic processes.

\section{Statement and Discussion of Results}
In this section, we collect and discuss the main results of the paper, the proofs of which may be found in later sections. Our first result establishes an upper bound on the intermediate dimensions of H\"older images using dimension profiles. Recalling that the $m$-intermediate dimension profiles intuitively tell us about the typical size of a set from an $m$-dimensional viewpoint for $m \in \{1, \dots, n\}$, it is interesting to note how the H\"older exponent dictates which profile appears in the bound. This is in contrast to the setting of projections \cite{bufafr:2019}, where the profile appearing in the upper-bound is simply the topological dimension of the codomain.

\begin{theorem}\label{holderthm}
Let $E \subset \R^n$ be compact, $\theta \in (0, 1]$, $ m \in \N$ and $f : E \rightarrow \R^m$. If there exists $c > 0$ and $0 < \alpha \leq 1$ such that 
\begin{equation}\label{holderlabel}
|f(x) - f(y)| \leq c|x-y|^\alpha
\end{equation}
for all $x, y \in E$, then
$$
\lid f(E) \leq \frac{1}{\alpha}\lid^{m\alpha}E 
$$
and
$$
\uid f(E) \leq \frac{1}{\alpha}\uid^{m\alpha}E.
$$
\end{theorem}

For certain families of mappings, such as fractional Brownian motion, we are able to obtain almost-sure lower bounds for the dimension of the images in terms of profiles too. Let $(\Omega, \mathcal{F}, P)$ denote a probability space with each $\omega \in \Omega$ corresponding to a $\sigma(\{F \times B : F \in \mathcal{F}, B \in \mathcal{B}\})$-measurable function $f_\omega : \R^n \rightarrow \R^m$, where $\mathcal{B}$ denotes the Borel subsets of $\R^n$. In order for this problem to be tractable, some condition must be placed on the set of functions. Specifically, we need to assume a relationship between
\begin{equation}\label{keyasum1}
\int 1_{[0, r]}(|f_\omega(x) - f_\omega(y)|) dP(\omega) = P\left(\left\{\omega : |f_\omega(x) - f_\omega(y)| \leq r\right\}\right)
\end{equation}
and the kernels (\ref{ker}). This is analogous to Matilla's result \cite[Lemma 3.11]{mat:book}, which covers the special case where $f_\omega$ denote orthogonal projections and $\Omega = G(n, m)$, the Grassmannian of $m$ dimensional subspaces of $\R^n$. However, such a result does not hold in general and so must be included as a hypothesis. This allows us to prove the following lemma that is a critical component of the following proofs. Essentially, it says that the integral of the modified kernels (\ref{modker}) over the probability space is bounded above by the kernels (\ref{ker}). \\

\begin{lemma}\label{hatlem}
Let $E \subset \R^n$ be compact, $\theta \in (0, 1]$, $\gamma > 0$, $m \in \N$ and $ 0 \leq s < m$. If $\{f_\omega : E \rightarrow \R^m, \omega \in \Omega\}$ is a set of continuous $\sigma(\{F \times B : F \in \mathcal{F}, B \in \mathcal{B}\})$-measurable functions such that there exists $c > 0$ satisfying
\begin{equation}\label{matanalog}
 P\left(\left\{\omega : |f_\omega(x) - f_\omega(y)| \leq r\right\}\right)\leq c\phi_{r^\gamma, \theta}^{m/\gamma, m/\gamma}(x - y) 
\end{equation}
 for all $x, y \in E$ and $r > 0$, then there exists $C_{s, m} > 0$ such that
$$ \int \widetilde{\phi}_{r, \theta}^{s}(f_\omega(x) - f_\omega(y))dP(\omega) \leq C_{s, m}\phi_{r^\gamma, \theta}^{s/\gamma, m/\gamma}(x- y).$$
\end{lemma}

This allows us to obtain the desired almost-sure lower bound. 

\begin{theorem}\label{ctsfam}
Let $E \subset \R^n$ be compact, $\theta \in (0, 1]$, $\gamma \geq 1$ and $m \in \N$. If $\{f_\omega : E \rightarrow \R^m, \omega \in \Omega\}$ is a set of continuous $\sigma(\{F \times B : F \in \mathcal{F}, B \in \mathcal{B}\})$-measurable functions such that there exists $c > 0$ satisfying
 \begin{equation}\label{condcts}
 P(\{\omega : |f_\omega(x) - f_\omega(y)| \leq r\}) \leq c\phi_{r^\gamma, \theta}^{m/\gamma, m/\gamma}(x - y)
 \end{equation}
for all $x, y \in E$ and $r > 0$, then
 $$
 \lid f_\omega(E) \geq \gamma \lid^{m/\gamma} E
 $$
 and
 $$
  \uid f_\omega(E) \geq \gamma \uid^{m/\gamma} E
 $$
 for $P$-almost all $\omega \in \Omega$.
\end{theorem}

Fractional Brownian motion is known to be $(\alpha - \varepsilon)$-H\"older and is shown to satisfy condition (\ref{condcts}) in Section \ref{frcsec}. Thus, a combination of Theorem \ref{holderthm} and Theorem \ref{ctsfam} yields our main result. 

\begin{theorem}\label{frcbrown}
Let $\theta \in (0,1]$, $m, n \in \N$, $B_\alpha : \R^n \rightarrow \R^m$ be index-$\alpha$ fractional Brownian motion $(0 < \alpha < 1)$ and $E \subset \R^n$ be compact. Then
$$
\lid B_\alpha(E) = \frac{1}{\alpha}\lid^{m\alpha} E
$$
and
$$
\uid B_\alpha(E) = \frac{1}{\alpha}\uid^{m\alpha} E
$$
almost surely.
\end{theorem}

Future work may take inspiration from \cite{wuxiao} and pursue uniform dimension results in this context for the intermediate dimensions. This could take the form of proving covering lemmas analogous to \cite[Lemma 3.2]{wuxiao} and \cite[Lemma 3.3]{wuxiao}.
\subsection{Observations and Applications}
Here we present a few applications of Theorems \ref{holderthm}, \ref{ctsfam} and \ref{frcbrown}, the proofs of which may be found in Section \ref{proofcorsec}.\\

Recent literature has sought to identify situations in which the intermediate dimensions are continuous at $\theta = 0$, for example, see \cite{bufafr:2019,fafrke:2018}. Theorem \ref{holderthm} implies that this continuity is preserved under index-$\alpha$ fractional Brownian motion. 
\begin{cor}\label{contcor}
Let $E \subset \R^n$ be bounded and $B_\alpha: \R^n \rightarrow \R^m$ denote index-$\alpha$ fractional Brownian motion with $m\alpha \leq n$. If $\lid E$ is continuous at $\theta = 0$, then $\lid B_\alpha (E)$ is almost surely continuous at $\theta = 0$. The analogous result holds for upper dimensions.
\end{cor}

Furthermore, Theorem \ref{holderthm} together with Corollary \ref{contcor} have a surprising application to the box and Hausdorff dimensions of sets with continuity at $\theta = 0$. In the following, we use the notation 
$$
\lbd^{n\alpha} E = \lidone^{n\alpha} E, 
$$
since our profiles extend the box dimension profiles $\lbd^m$ of Falconer \cite{fal:2019} to non-integer values of $m$ when $\theta = 1$ (and similarly for the upper dimensions). 

\begin{cor}\label{surprise}
Let $E \subset \R^n$ be a bounded set such that $\lid E$ is continuous at $\theta = 0$. If $\alpha > \frac{1}{n}\hd E$, then
$$
\frac{1}{\alpha}\lbd^{n\alpha} E < n.
$$
On the other hand, if $\alpha \leq \frac{1}{n}\hd E$, then
$$
\frac{1}{\alpha}\lbd^{n\alpha} E = n.
$$
The analogous result holds for upper dimensions.
\end{cor}
In particular, since $\hd E \leq \lbd E$, the first part of Corollary \ref{surprise} shows us that $\lbd^{n\alpha} E$ is strictly less than the trivial upper bound of $n\alpha$ implied by Lemma \ref{existlem} for
$$
\alpha \in \left(\frac{\hd E}{n}, \frac{\lbd E}{n}\right),
$$
and similarly for $\ubd E$. Furthermore, Corollary \ref{surprise} may immediately be translated into the context of fractional Brownian motion by Theorem \ref{frcbrown}. 
\begin{cor}\label{surprisebrownian}
Let $E \subset \R^n$ be a bounded set such that $\lid E$ is continuous at $\theta = 0$ and $B_\alpha : \R^n \rightarrow \R^n$ denote index-$\alpha$ Brownian motion. If $\alpha > \frac{1}{n}\hd E$, then
$$
\lbd B_\alpha(E) < n.
$$
almost surely. On the other hand, if $\alpha \leq \frac{1}{n}\hd E$, then
$$
\lbd B_\alpha(E) = n.
$$
almost surely. The analogous result holds for upper dimensions.
\end{cor}
It may be of interest to see how Corollary \ref{surprisebrownian}, which deals with box dimension, differs from the related classical result of Kahane on the Hausdorff dimensions of Brownian images \cite[Corollary, pp. 267]{kahane:book}.\\

A further implication of Theorem \ref{frcbrown} is that an inequality derived from a slight modification of the proof allows us to show in Section \ref{contproof} that the dimension profiles are continuous for any Borel set $E \subseteq \R^n$. 

\begin{cor}\label{cont}
Let $E \subseteq \R^n$ be bounded and $\theta \in (0,1]$. The functions $f, g : (0, n) \rightarrow [0, n]$ defined by
$$
f(t) = \lid^t E
$$
and
$$
g(t) \rightarrow \uid^t E
$$
are continuous in $t$. 
\end{cor}

One final application concerns the Hausdorff dimension of the set of exceptional sets in the projection setting. The proof is based on an application of Theorem \ref{ctsfam}, which allows the proof of \cite[Theorem 1.2 (ii), (iii)]{fal:2019} to be generalised from box dimension (the case where $\theta = 1$) to all intermediate dimensions. 

\begin{theorem}\label{excep}
Let $E \subset \R^n$ be compact, $m \in \{1, \dots, n\}$ and $0 \leq \lambda \leq m$, then 
\begin{equation}
\hd \{V \in G(n, m) : \uid \pi_V E < \uid^\lambda E\} \leq m(n-m) - (m - \lambda)
\end{equation}
and
\begin{equation}
\hd \{V \in G(n, m) : \lid \pi_V E < \lid^\lambda E\} \leq m(n-m) - (m - \lambda)
\end{equation}
\end{theorem}

Recall that $\uid^\lambda E$ and $\lid^\lambda E$ decrease as $\lambda$ decreases. Thus, Theorem \ref{excep} tells us that the there is a stricter upper bound on the dimension of the exceptional set the larger the drop in dimension from the expected value. We conclude by posing a slightly different question which is a slight strengthening of Theorem \ref{excep}, an analogy of which  was considered in \cite[Theorem 1.3 (ii), (iii)]{fal:2019}.
\begin{question}
Let $0 \leq \gamma \leq n - m$. What are the optimum upper bounds for 
$$
\hd \{V \in G(n, m) : \uid \pi_V E < \uid^{m+\gamma} E - \gamma\} 
$$
and
$$
\hd \{V \in G(n, m) : \lid \pi_V E < \lid^{m + \gamma}  E- \gamma \}?
$$
\end{question}
The method in \cite{fal:2018} for box dimensions relied on Fourier transforms and approximating the potential kernels by a Gaussian with a strictly positive Fourier transform. However, the natural family of kernels appropriate for working with intermediate dimension have a more complex shape, which complicates matters. A significantly different, but perhaps interesting, approach may be required.
\section{Proof of Theorem \ref{holderthm}}
To prove Theorem \ref{holderthm} we use the following result \cite[Lemma 4.4]{bufafr:2019}, which is stated here for convenience. 
\begin{lemma}\label{capacityub}
Let $E\subset \R^m$ be compact, $0 \leq s \leq m$ and  $\theta \in (0, 1]$. If there exists a measure $\mu \in \mathcal{M}(E)$ and $\beta > 0$ such that
\begin{equation}\label{cond}
\int \phi_{r, \theta}^{s, m}(x- y) d\mu(y) \geq \beta
\end{equation}
for all $x \in E$, then there is a number $r_0>0$ such that for all $0<r\leq r_0$,
$$ S_{r,\theta}^s(E) \leq a_m\ceil{\log_2(|E|/r)+1}\frac{r^{s}}{\beta}$$
where the constant $a_m$ depends only on $m$. In particular, 
$$ S_{r,\theta}^s(E) \leq a_m\ceil{\log_2(|E|/r)+1} C_{r,\theta}^{s, m}(E)r^s.$$
\end{lemma}

Intermediate dimension is invariant under scaling and thus we may assume the H\"older constant $c$ in (\ref{holderlabel}) equals one. First, note
$$
\frac{r^s}{|x - y|^{\alpha s}} \leq \frac{r^{\theta(m - s) + s}}{|x - y|^{\alpha m}}
$$
for $|x - y| \leq r^{\theta /\alpha}$ and $0 \leq s \leq m$. It then follows from the definition of $\phi_{r, \theta}^{s, m}$ that
\begin{align*}
\phi_{r, \theta}^{s, m}(f(x) - f(y))& = \min\left\{1, \frac{r^s}{|f(x) - f(y)|^s}, \frac{r^{\theta(m - s) + s}}{|f(x) - f(y)|^m}\right\}\\
&\geq \min\left\{1, \frac{r^s}{|x - y|^{\alpha s}}, \frac{r^{\theta(m - s) + s}}{|x - y|^{\alpha m}}\right\}\\
&=  \begin{cases}
1 & |x-y| < r^{1/\alpha}\\
\left({r^{1/\alpha}}/{|x-y|}\right)^{s\alpha} &r^{1/\alpha} \leq |x-y| \leq r^{\theta/\alpha}\\
(r^{1/\alpha})^{\theta(m\alpha -s\alpha) + s\alpha}/\left(|x-y|\right)^{m\alpha} & |x-y| > r^{\theta/\alpha}
\end{cases}\\
&= \phi_{r^{1/\alpha}, \theta}^{s\alpha, m\alpha}(x- y).
\end{align*}
By Lemma \ref{attaincap}, for each $0 \leq s \leq m$ there exists a measure $\mu \in \mathcal{M}(E)$ such that for all $x \in E$
\begin{align*}
\frac{1}{C_{r^{1/\alpha}, \theta}^{s\alpha, m\alpha}(E)} &\leq \int \phi_{r^{1/\alpha}, \theta}^{s\alpha, m\alpha}(x- y) d\mu(y) \\
&\leq \int \phi_{r, \theta}^{s, m}(f(x) - f(y)) d\mu(y) \\
& \leq \int \phi_{r, \theta}^{s, m}(f(x) - w)d(f\mu)(w)
\end{align*}
where $f\mu \in \mathcal{M}(E)$ is defined by $\int g(w)d(f\mu)(w) = \int g(f(x))d\mu(x)$ for all continuous functions $g$ and by extension. This verifies that $f(E)$ supports a measure satisfying the condition of Lemma \ref{capacityub}. Hence, for sufficiently small $r > 0$, 
$$
S_{r, \theta}^{s}(f(E)) \leq a_m \ceil{\log_2 (|E|/r)+1}r^{s} C_{r^{1/\alpha}, \theta}^{s\alpha, m\alpha}(E)
$$
for all $0 \leq s \leq m$. This implies
$$
\liminf\limits_{r \rightarrow 0}\frac{S_{r, \theta}^s(f(E))}{-\log r} \leq -s+ \liminf\limits_{r \rightarrow 0}\frac{C_{r^{1/\alpha}, \theta}^{s\alpha, m\alpha}(E)}{-\alpha\log r^{1/\alpha}},
$$
and so 
$$
\alpha\liminf\limits_{r \rightarrow 0}\frac{S_{r, \theta}^s(f(E))}{-\log r} \leq -s\alpha+ \liminf\limits_{r \rightarrow 0}\frac{C_{r^{1/\alpha}, \theta}^{s\alpha, m\alpha}(E)}{-\log r^{1/\alpha}}.
$$
Recall,
$$
\frac{1}{\alpha}\lid^{m\alpha} E \leq \frac{1}{\alpha}m\alpha = m.
$$
and thus we may set $s\alpha = \lid^{m\alpha} E$. It follows
$$
\liminf\limits_{r \rightarrow 0}\frac{S_{r, \theta}^{\frac{1}{\alpha} \lid^{m\alpha} E}(f(E))}{-\log r} \leq 0,
$$
implying
$$\lid f(E) \leq \frac{1}{\alpha} \lid^{m\alpha} E.$$
The inequality for $\uid f(E)$ follows by using a similar argument and taking upper limits. $\square$

\section{Proof of Lemma \ref{hatlem} and Theorem \ref{ctsfam}}

\subsection{Proof of Lemma \ref{hatlem}}
Let $\theta \in (0,1]$. To ease notation, define
$$
\phi_{r^\gamma}^{m/\gamma}(x-y) := \phi_{r^\gamma, \theta}^{m/\gamma, m/\gamma}(x - y) = 
\begin{cases} 
      1 & |x - y| < r^\gamma \\
       \left(\frac{r^\gamma}{|x-y|}\right)^{m/\gamma}  & |x-y| \geq r^\gamma
   \end{cases},
$$
since the kernels $\phi_{r, \theta}^{s, t}$ lose dependence on $\theta$ and take the same form on $[r,r^\theta]$ and $(r^\theta, \infty)$ when $s = t$.\\

Recall, from \cite[Lemma 5.3]{bufafr:2019}, that
$$
\widetilde{\phi}_{r, \theta}^{s}(x) = sr^s\int\limits_{u = r}^{r^\theta} 1_{[0, u]}(|x|)u^{-(s+1)}du + r^{s(1-\theta)}1_{[0, r^\theta]}(|x|),
$$
and so by Fubini's theorem
\begin{align*} 
\int \widetilde{\phi}_{r, \theta}^s(f_\omega(x) - f_\omega(y)) dP(\omega) &=
 sr^s \int\limits_{u = r}^{r^\theta} u^{-(s+1)} \left[\int 1_{[0, u]}(|f_\omega(x) - f_\omega(y)|)dP(\omega)\right]\,du\\ &\,\,\,\,\,\,+ r^{s(1-\theta)}\int 1_{[0, r^\theta]}(|f_\omega(x) - f_\omega(y)|)dP(\omega).
\end{align*}
From (\ref{matanalog}), 
\begin{equation}
\int 1_{[0, u]}(|f_\omega(x) - f_\omega(y)|) dP(\omega) \leq c \phi_{u^\gamma}^{m/\gamma}(x - y)
\end{equation}
and
\begin{equation}
\int 1_{[0, r^\theta]}(|f_\omega(x) - f_\omega(y)|) dP(\omega) \leq c \phi_{r^{\theta\gamma}}^{m/\gamma}(x - y).
\end{equation}
Hence
$$
\frac{1}{c}\int \widetilde{\phi}_{r, \theta}^s(f_\omega(x) - f_\omega(y)) dP(\omega) \leq sr^s \int\limits_{u = r}^{r^\theta} u^{-(s+1)}\phi_{u^\gamma}^{m/\gamma}(x - y)\,du + r^{s(1-\theta)}\phi_{r^{\theta\gamma}}^{m/\gamma}(x - y),
$$
which must be evaluated in three cases.\\
\vspace{0.6cm}\\

\textbf{Case 1: } Suppose $|x - y| \leq r^\gamma$, then
$$
\phi_{u^\gamma}^{m/\gamma}(x - y) = 1
$$
for all $r \leq u \leq r^\theta$, and
$$
\phi_{r^{\theta\gamma}}^{m/\gamma}(x -y) = 1.
$$
Hence
\begin{align*}
\frac{1}{c}\int \widetilde{\phi}_{r, \theta}^s(f_\omega(x) - f_\omega(y)) dP(\omega) &\leq sr^s \int\limits_{u = r}^{r^\theta} u^{-(s+1)}\phi_{u^\gamma}^{m/\gamma}(x - y)\,du + r^{s(1-\theta)}\phi_{r^{\theta\gamma}}^{m/\gamma}(x - y)\\
&= sr^s \int\limits_{u = r}^{r^\theta} u^{-(s+1)}\,du + r^{s(1-\theta)}\\
&= 1.
\end{align*}

\textbf{Case 2:} Suppose $r^\gamma \leq |x - y| \leq r^{\theta \gamma}$, then
$$
\phi_{r^{\theta\gamma}}^{m/\gamma}(x -y) = 1.
$$
Moreover, for $r \leq u \leq |x - y|^{1/\gamma}$ we have
$$
\phi_{u^\gamma}^{m/\gamma}(x-y) = \frac{u^{m}}{|x - y|^{m/\gamma}}
$$
and 
$$
\phi_{u^\gamma}^{m/\gamma}(x-y) = 1
$$
for $|x-y|^{1/\gamma} \leq u \leq r^\theta$. Hence
\begin{align*}
&\frac{1}{c}\int \widetilde{\phi}_{r, \theta}^s(f_\omega(x) - f_\omega(y)) dP(\omega) \\
&\leq sr^s \int\limits_{u = r}^{r^\theta} u^{-(s+1)}\phi_{u^\gamma}^{m/\gamma}(x - y)\,du + r^{s(1-\theta)}\phi_{r^{\theta\gamma}}^{m/\gamma}(x - y)\\
&= sr^s \int\limits_{u = r}^{r^\theta} u^{-(s+1)}\phi_{u^\gamma}^{m/\gamma}(x - y)\,du + r^{s(1-\theta)}\\
&= sr^s \int\limits_{u = r}^{|x - y|^{1/\gamma}}u^{-(s+1)}\frac{u^{m}}{|x - y|^{m/\gamma}}\,du
+ sr^s \int\limits_{u = |x - y|^{1/\gamma}}^{r^\theta} u^{-(s + 1)}\,du + r^{s(1-\theta)}\\
&\leq \left(\frac{s}{m - s} + 1\right)\left(\frac{r^\gamma}{|x-y|}\right)^{s/\gamma}.
\end{align*}
\vspace{0.6cm}\\

\textbf{Case 3:} Suppose $|x - y| \geq r^{\theta \gamma}$, then
$$
\phi_{r^{\theta\gamma}}^{m/\gamma}(x -y) = \frac{r^{\theta m}}{|x- y|^{m/\gamma}}
$$
and
$$
\phi_{u^\gamma}^{m/\gamma}(x-y) = \frac{u^{m}}{|x - y|^{m/\gamma}}
$$
for $r \leq u \leq r^\theta$. \\

Hence
\begin{align*}
&\frac{1}{c}\int \widetilde{\phi}_{r, \theta}^s(f_\omega(x) - f_\omega(y)) dP(\omega) \\
&\leq sr^s \int\limits_{u = r}^{r^\theta} u^{-(s+1)}\phi_{u^\gamma}^{m/\gamma}(x - y)\,du + r^{s(1-\theta)}\phi_{r^{\theta\gamma}}^{m/\gamma}(x - y)\\
&= sr^s \int\limits_{u = r}^{r^\theta} u^{-(s+1)} \frac{u^{m}}{|x - y|^{m/\gamma}}\,du + r^{s(1-\theta)}\frac{r^{\theta m}}{|x- y|^{m/\gamma}}\\
&= \left(\frac{s}{m - s} + 1\right)\frac{(r^{\gamma})^{\theta(m/\gamma - s/\gamma) + s/\gamma}}{|x- y|^{m/\gamma}}.\\
\end{align*}
To conclude, we deduce from Case 1, Case 2 and Case 3 that
\begin{align*}
&\frac{1}{c}\int \widetilde{\phi}_{r, \theta}^s(f_\omega(x) - f_\omega(y)) dP(\omega) \\
&\leq \begin{cases} 
      1 & |x - y| < r^\gamma \\
       \left(\frac{s}{m - s} + 1\right)\left(\frac{r^\gamma}{|x-y|}\right)^{s/\gamma}  & r^\gamma\leq |x - y| \leq r^{\gamma\theta} \\
      \left(\frac{s}{m - s} + 1\right)\frac{(r^{\gamma})^{\theta(m/\gamma - s/\gamma) + s/\gamma}}{|x- y|^{m/\gamma}} & r^{\gamma\theta} < |x - y| 
   \end{cases} \\
  & \leq \left(\frac{s}{m-s} + 1\right) \phi_{r^\gamma, \theta}^{s/\gamma, m/\gamma}(x -y),
\end{align*}
as required. $\square$
\subsection{Proof of Theorem \ref{ctsfam}}
Let $E \subset \R^n$ be compact, $\theta \in (0, 1]$, $\gamma \geq 1$, $m \in \N$ and $0 \leq s < m$. Choose a sequence $(r_k)_{k \in \N}$ such that $0 < r_k < 2^{-k}$ and 
\begin{equation}\label{supest}
\limsup\limits_{k \rightarrow \infty} \frac{C_{r_k^\gamma, \theta}^{s, m}(E)}{-\log r_k^\gamma} = \limsup\limits_{r \rightarrow 0} \frac{C_{r, \theta}^{s, m}(E)}{-\log r}.
\end{equation}
Moreover, define a sequence of constants $\beta_k$ by
$$
\beta_k := \frac{1}{C_{r_k^\gamma, \theta}^{s/\gamma, m/\gamma}(E)} = \int\int \phi_{r_k^\gamma, \theta}^{s/\gamma, m/\gamma}(x- y) d\mu^k(x)\mu^k(y),
$$
where $\mu^k$ is the equilibrium measure from Lemma \ref{attaincap} on $E$ associated with the kernel $\phi_{r_k^\gamma, \theta}^{s/\gamma, m/\gamma}$. \\

Hence, by (\ref{condcts}) and Lemma \ref{hatlem} we have
\begin{align*}
&\int\int\int \widetilde{\phi}_{r_k, \theta}^{s}(f_\omega (x) - f_\omega(y)) dP(\omega) d\mu^k(x)d\mu^k(y)\\
&\leq C_{s, m}\int\int \phi_{r_k^\gamma, \theta}^{s/\gamma, m/\gamma}(x- y)d\mu^k(x)d\mu^k(y) \\
&\leq C_{s, m} \beta_k.
\end{align*}
Then, for each $\varepsilon > 0$,
$$
\int\int\int \beta_k^{-1}r_k^{\varepsilon}\widetilde{\phi}_{r_k, \theta}^{s}(f_\omega (x) - f_\omega(y)) dP(\omega) d\mu^k(x)d\mu^k(y) \leq C_{s, m} r_k^{\varepsilon}
$$
from which Fubini's theorem implies
$$
\int \sum\limits_{k =1}^{\infty} \left(\int\int\beta_k^{-1}r_k^{\varepsilon}\widetilde{\phi}_{r_k, \theta}^{s}(f_\omega (x) - f_\omega(y))  d\mu^k(x)d\mu^k(y) \right)dP(\omega) \leq C_{s, m} \sum\limits_{k =1}^{\infty}r_k^{\varepsilon} < \infty
$$
since $|r_k^{\varepsilon}| \leq 2^{-k\varepsilon}$. Hence, for $P$-almost all $\omega \in \Omega$, there exists $M_\omega > 0$ such that
$$
\int\int\beta_k^{-1}r_k^{\varepsilon}\widetilde{\phi}_{r_k, \theta}^{s}(t - u)  d\mu^k_\omega(t)d\mu^k_\omega(u) \leq M_\omega < \infty
$$
for all $k$, where $\mu^k_\omega$ is the image of $\mu^k$ under $f_\omega$. Thus,
$$
\int\int\widetilde{\phi}_{r_k, \theta}^{s}(t - u)  d\mu^k_\omega(t)d\mu^k_\omega(u) \leq M_\omega\beta_k r_k^{-\varepsilon} 
$$
for all $k$. Hence, for each $k$ there exists a set $F_k \subset f_\omega(E)$ with $\mu^k_\omega(F_k) \geq 1/2$ and
$$
\int \widetilde{\phi}_{r^k, \theta}^{s}(t- u)d\mu^k_\omega(t) \leq 2M_\omega \beta_k r_k^{-\varepsilon}
$$
for all $u \in F_k$. Hence, by \cite[Lemma 5.4]{bufafr:2019}
$$
S_{r_k, \theta}^{s}(f_\omega(E)) \geq \frac{1}{2}(2M_\omega\beta_k)^{-1}r_k^{s + \varepsilon} = (4M_\omega\beta_k)^{-1}r_k^{s+ \varepsilon},
$$
and so
\begin{align*}
\uplim\limits_{k \rightarrow \infty}\frac{\log S_{r_k, \theta}^{s}(f_\omega(E))}{-\log r_k} &\geq \uplim\limits_{k \rightarrow \infty}\frac{\log r_k^{s + \varepsilon}(4M_\omega\beta_k)^{-1}}{-\log r_k}\\
&= \uplim\limits_{k \rightarrow \infty}\frac{\log r_k^{s + \varepsilon}C_{r_k^\gamma, \theta}^{s/\gamma, m/\gamma}(E)}{-\log r_k}\\
&= -(s + \epsilon) + \uplim\limits_{k \rightarrow \infty} \frac{\log C_{r_k^\gamma, \theta}^{s/\gamma, m/\gamma}(E)}{-\log r_k}.
\end{align*}
Hence
\begin{align*}
\frac{1}{\gamma}\uplim\limits_{k \rightarrow \infty}\frac{\log S_{r_k, \theta}^{s}(f_\omega(E))}{-\log r_k} 
&\geq -\frac{s + \varepsilon}{\gamma}+ \uplim\limits_{k \rightarrow \infty} \frac{\log C_{r_k^\gamma, \theta}^{s/\gamma, m/\gamma}(E)}{-\log r_k^\gamma}.
\end{align*}
This is true for all $\epsilon >0$, so using (\ref{supest}),
$$
\frac{1}{\gamma}\uplim\limits_{r\to 0}\frac{\log S_{r, \theta}^{s}(f_\omega(E))}{-\log r} 
\geq
-\frac{s}{\gamma}  + \uplim\limits_{r\to 0} \frac{\log C_{r, \theta}^{s/\gamma, m/\gamma}(E)}{-\log r}
$$
 for all $s \in [0,m)$. Since the expressions on both sides of this inequality are continuous for $s \in [0,m]$  by \cite[Lemma 2.1]{bufafr:2019} and \cite[Lemma 3.2]{bufafr:2019}, the inequality is valid for $s \in [0,m]$ and consequently $s/\gamma \in [0, m/\gamma]$. Hence, for $s/\gamma = \uid^{m/\gamma} E$
 $$
\uplim\limits_{r\to 0}\frac{\log S_{r, \theta}^{s}(f_\omega(E))}{-\log r} 
\geq
0,
$$
implying $\uid f_\omega(E) \geq s = \gamma \uid^{m/\gamma} E$. The argument for $\lid f_\omega E$ is similar, although it suffices to set $r_k = 2^{-k}$. $\square$

\section{Proof of Theorem \ref{frcbrown}}\label{frcsec}
Let $\theta \in (0, 1]$ and $0 < \varepsilon < \alpha < 1$. By \cite[Corollary 2.11]{fal:2018} there exists, almost surely, $M > 0$ such that
\begin{equation}\label{holdcont}
|B_\alpha(x) - B_\alpha(y)| \leq M|x- y|^{\alpha - \varepsilon}
\end{equation}
for all $x, y \in E$. In addition,
\begin{align}\label{probcalc}
\mathbb{P}(|B_\alpha(x)-B_\alpha(y)| \leq r) &\leq \mathbb{P}(|B_{\alpha, i}(x)- B_{\alpha, i}(y)| \leq r \textnormal{ for all $1 \leq i \leq m$})\nonumber\\
&\leq \left(\frac{1}{\sqrt{2\pi}}\frac{1}{|x-y|^\alpha}\int\limits_{|t| \leq r} \exp\left(\frac{-t^2}{2|x-y|^{2\alpha}}\right)\,dt\right)^m\nonumber\\
&\leq \left(\frac{1}{|x-y|^\alpha}\int\limits_{|t| \leq r} 1\,dt\right)^m\nonumber\\
&= 2^m \left(\frac{r^{1/\alpha}}{|x-y|}\right)^{m\alpha}\\
&\leq 2^{\max\{m, n\}}\phi_{r^{\gamma}, \theta}^{m/\gamma , m/\gamma} (x -y)\nonumber
\end{align}
for all $x, y\in E$ and $r > 0$, where $\gamma = 1/\alpha$. \\

By applying Theorem \ref{ctsfam} and Theorem \ref{holderthm},
$$
\frac{1}{\alpha}\lid^{m\alpha} E \leq \lid B_\alpha(E) \leq \frac{1}{\alpha - \varepsilon}\lid^{m(\alpha-\varepsilon)}E \leq \frac{1}{\alpha - \varepsilon}\lid^{m\alpha}E
$$
and 
$$
\frac{1}{\alpha}\uid^{m\alpha} E \leq \uid B_\alpha(E) \leq \frac{1}{\alpha - \varepsilon}\uid^{m(\alpha-\varepsilon)}E \leq \frac{1}{\alpha - \varepsilon}\uid^{m\alpha}E
$$
almost surely, with the last inequality in each case holding since the profiles are monotonically increasing. Letting $\varepsilon \rightarrow 0$, the result follows. $\square$
\section{Proof of Corollaries \ref{contcor}, \ref{surprise} and \ref{cont}}\label{proofcorsec}
\subsection{Proof of Corollary \ref{contcor}}
From \cite[Corollary, pp. 267]{kahane:book}, 
$$
\hd B_\alpha(E) = \frac{1}{\alpha} \hd E
$$
almost surely, and so
$$
\hd E \leq \alpha \lid B_\alpha(E) \leq \alpha\frac{1}{\alpha}\lid^{m\alpha} E \leq \lid^{n} E = \lid E
$$
by monotonicity of the profiles. Hence, as $\theta \rightarrow 0$, continuity of $\lid B_\alpha(E)$ at $\theta = 0$ is established, since $\lid E \rightarrow \hd E$ by definition. The proof for upper dimensions is similar.$\square$

\subsection{Proof of Corollary \ref{surprise}}
Let $E \subset \R^n$ be such that $\lid E$ is continuous at $\theta = 0$, and let $B_\alpha : \R^n \rightarrow \R^n$ denote index-$\alpha$ fractional Brownian motion where
$$
\alpha > \frac{\hd E}{n}.
$$
Hence, by \cite[Corollary, pp. 267]{kahane:book}, 
\begin{equation}\label{contradict}
\hd B_\alpha(E) = \frac{1}{\alpha} \hd E < n
\end{equation}
almost surely. Then, in order to reach a contradiction, let us suppose that $\lbd^{n\alpha} E = n\alpha$. This implies $\lbd B_\alpha(E) = n$ almost surely by Theorem \ref{frcbrown}. Then, by \cite[Corollary 6.3]{bufafr:2019},
$$
\lid B_\alpha(E) = n
$$
almost surely for all $\theta \in (0, 1]$. By Corollary \ref{contcor}, $\lid B_\alpha(E)$ is continuous at $\theta = 0$ which implies $\hd B_\alpha(E) = n$, a contradiction to (\ref{contradict}). The case for $\alpha \leq \frac{1}{n}\hd E$ follows easily from \cite[Corollary, pp. 267]{kahane:book} and Theorem \ref{frcbrown}. $\square$

\subsection{Proof of Corollary \ref{cont}}\label{contproof}
Let $0 < s < n$ and $\theta \in (0, 1]$.  Fix $\alpha > 0$ such that $n\alpha = s$. Since $E$ is bounded, there exists $B > 1$ such that
\begin{equation*}
|x-y| < B
\end{equation*}
for all $x, y \in E$. Let $\varepsilon >0$ be such that $n(\alpha + \varepsilon)/(1 - \varepsilon) < n$, and choose $C_\varepsilon \geq B^{\varepsilon(1 + \alpha)/(1-\varepsilon)}$. Observe
\begin{align}\label{xybound}
C_\varepsilon &\geq |x -y|^{\varepsilon(1 + \alpha)/(1-\varepsilon)} \nonumber\\
& = \frac{|x-y|^{(\alpha+\varepsilon)/(1-\varepsilon)}}{|x-y|^{\alpha}}
\end{align}
for all $x, y \in E$. Then, consider $B_\alpha : \R^n \rightarrow \R^n$. By (\ref{probcalc}) and (\ref{xybound}),
\begin{align}
\mathbb{P}(|B_\alpha(x)-B_\alpha(y)| \leq r) &\leq 2^n \min\left\{1, \left(\frac{r^{1/\alpha}}{|x-y|}\right)^{n\alpha}\right\}\nonumber\\
&\leq 2^nC_\varepsilon^n\min\left\{1, \left(\frac{r^{(1-\varepsilon)/(\alpha + \varepsilon)}}{|x-y|}\right)^{n\frac{\alpha + \varepsilon}{1-\varepsilon}}\right\}\nonumber\nonumber\\
&= 2^nC_\varepsilon^n\phi_{r^{\gamma}, \theta}^{n/\gamma , n/\gamma} (x -y)\nonumber
\end{align}
for all $x, y\in E$ and $r > 0$, where $\gamma = (1-\varepsilon)/(\alpha + \varepsilon)$. Hence, from Theorem \ref{holderthm} and Theorem \ref{ctsfam}, we have

\begin{equation*}
\frac{1 - \varepsilon}{\alpha + \varepsilon} \lid^{n({\alpha + \varepsilon})/({1 - \varepsilon})} E \leq \lid B_\alpha(E) \leq \frac{1}{\alpha - \varepsilon}\lid^{n(\alpha - \varepsilon)}E
\end{equation*}
almost surely. The profiles are monotonically increasing, and so
\begin{equation*}
\frac{1 - \varepsilon}{\alpha + \varepsilon} \lid^{s} E \leq \frac{1 - \varepsilon}{\alpha + \varepsilon} \lid^{n({\alpha + \varepsilon})/({1 - \varepsilon})} E \leq \frac{1}{\alpha}\lid^s E \leq \frac{1}{\alpha - \varepsilon}\lid^{n(\alpha - \varepsilon)}E \leq \frac{1}{\alpha - \varepsilon} \lid^{s} E
\end{equation*}
almost surely, since 
$$
\frac{n(\alpha + \varepsilon)}{1-\varepsilon} > s > n(\alpha - \varepsilon).
$$
This holds for arbitrary sequences of sufficiently small positive $\varepsilon$ tending to zero and so establishes continuity from above and below. The proof for $\uid^s$ is similar. $\square$
\section{Proof of Theorem \ref{excep}}
First, define 
$$
A = \{ V \in G(n, m) : \uid \pi_V E < \uid^\lambda E\}
$$
and suppose, with the aim of deriving a contradiction, that
$$
\hd A > m(n - m) - (m - \lambda).
$$

By Frostman's lemma, there exists a measure $\mu$ supported on a compact set $B \subseteq A$ and $c > 0$ such that
$$
\mu(B_G(V, r)) \leq cr^{m(n-m) - (m - \lambda)}
$$
for all $V \in G(n, m)$ and $r > 0$, where $B_G$ is a ball defined via the natural metric of dimension $m(n - m)$ on $G(n, m)$. Hence, using \cite[Inequality (5.12)]{mat:book2} yields
\begin{align*}
\mu(\{V \in G(n, m) : |\pi_V x - \pi_V y | < r\}) &\leq \left(\frac{r}{|x - y|}\right)^{m(n - m) - (m - \lambda) - m(n-m-1)}\\
&= \left(\frac{r}{|x - y|}\right)^{\lambda} \\
&\leq \phi_{r, \theta}^{\lambda, \lambda}(x -y ).
\end{align*}
Thus, the condition of Theorem \ref{ctsfam} is satisfied with $\Omega = G(n, m)$, $P = \mu$ and $\gamma = m/\lambda$. Hence
\begin{equation}\label{above}
\uid \pi_V E \geq \uid^{\lambda} E 
\end{equation}
for $\mu$ almost-all $V \in G(n, m)$. Since $\mu$ is supported on $A$, this is a contradiction, as it implies the existence of $V \in A$ satisfying (\ref{above}). The proof for $\lid$ follows similarly. $\square$
\section{Acknowledgement}
The author thanks the Carnegie Trust and London Mathematical Society for financially supporting this work, as well as Kenneth Falconer and Jonathan Fraser for insightful discussion and comments.

\end{document}